\documentclass{amsart} 
\usepackage{amssymb}
\usepackage{amsmath}
\usepackage{amsfonts}

\begin{document}
\newtheorem{theo}{Theorem}[section]
\newtheorem{atheo}{Theorem*}
\newtheorem{prop}[theo]{Proposition}
\newtheorem{aprop}[atheo]{Proposition*}
\newtheorem{lemma}[theo]{Lemma}
\newtheorem{alemma}[atheo]{Lemma*}
\newtheorem{exam}[theo]{Example}
\newtheorem{coro}[theo]{Corollary}
\theoremstyle{definition}
\newtheorem{defi}[theo]{Definition}
\newtheorem{rem}[theo]{Remark}


\newcommand{\Bb}{{\bf B}}
\newcommand{\Cb}{{\mathbb C}}
\newcommand{\Nb}{{\mathbb N}}
\newcommand{\Qb}{{\mathbb Q}}
\newcommand{\Rb}{{\mathbb R}}
\newcommand{\Zb}{{\mathbb Z}}
\newcommand{\Ac}{{\mathcal A}}
\newcommand{\Bc}{{\mathcal B}}
\newcommand{\Cc}{{\mathcal C}}
\newcommand{\Dc}{{\mathcal D}}
\newcommand{\Fc}{{\mathcal F}}
\newcommand{\Ic}{{\mathcal I}}
\newcommand{\Jc}{{\mathcal J}}
\newcommand{\Kc}{{\mathcal K}}
\newcommand{\Lc}{{\mathcal L}}
\newcommand{\Oc}{{\mathcal O}}
\newcommand{\Pc}{{\mathcal P}}
\newcommand{\Sc}{{\mathcal S}}
\newcommand{\Tc}{{\mathcal T}}
\newcommand{\Uc}{{\mathcal U}}
\newcommand{\Vc}{{\mathcal V}}

\author{Nik Weaver}

\title [Unique preduals for Lipschitz spaces]
       {On the unique predual problem for Lipschitz spaces}


\begin{abstract}
For any metric space $X$, the predual of ${\rm Lip}(X)$ is unique.

A previous version of this manuscript, which is also the published version ({\it Math.\ Proc.\ Cambridge Philos.\ Soc.\ \bf 165} (2018), 467–473), additionally stated ``If
$X$ has finite diameter or is complete and convex --- in particular, if it
is a Banach space --- then the predual of ${\rm Lip}_0(X)$ is unique.'' However, the proof of a crucial lemma, Lemma 3.1 in the previous version, was faulty. The error in that proof lay in assuming that the limit, for the topology induced by $\tilde{W}$, of a net in the unit ball would have to lie in the unit ball. But we do not know that $\tilde{W}$ is 1-norming.

This error was pointed out by Manuel González, as relayed to me by Rubén Medina.

The reduction from ``complete and convex'' to ``finite diameter'' is still valid, and is retained in the present version.
\end{abstract}

\maketitle

\section{Lipschitz spaces}

Let $(X,\rho)$ be a metric space. The {\it Lipschitz number} of a function
$f: X \to \mathbb{R}$ is the quantity
$$L(f) = \sup_{p \neq q} \frac{|f(p) - f(q)|}{\rho(p,q)};$$
$f$ is {\it Lipschitz} if $L(f) < \infty$. Note that the Lipschitz
number of any constant function is zero.
We define ${\rm Lip}(X)$ to be the space of bounded Lipschitz functions
from $X$ into $\mathbb{R}$, equipped with the norm
$\|f\|_L = \max(L(f),\|f\|_\infty)$.

If $X$ is a pointed metric space, i.e., it is equipped with a distinguished
``base point'' $e$, then we define ${\rm Lip}_0(X)$ to be the space of all
(possibly unbounded)
Lipschitz functions from $X$ into $\mathbb{R}$ which vanish at $e$, with
norm $L(\cdot)$. Since the zero function is the only constant function
which vanishes at the base point, Lipschitz number is a norm on this space,
not merely a seminorm. The Banach space structure of
${\rm Lip}_0(X)$ does not depend on the choice of $e$; if $e'$ is another
base point then the map $f \mapsto f - f(e')\cdot 1_X$ is an isometric
linear map taking the Lipschitz functions which vanish at $e$ to the
Lipschitz functions which vanish at $e'$.

${\rm Lip}$ and ${\rm Lip}_0$ spaces are closely related. Write $\cong$ for
isometric isomorphism.

\begin{prop}\label{lipvslip0}
(\cite[Proposition 1.7.1, Theorem 1.7.2, and Proposition 1.7.3]{W})
Given any metric space $X$, let $Y$ be the pointed set $X \cup \{e\}$
equipped with the metric which sets
$$\rho^Y(p,q) = \min(\rho^X(p,q),2)$$
for all $p,q \in X$ and $\rho^Y(p,e) = 1$ for all $p \in X$. Then
${\rm Lip}(X) \cong {\rm Lip}_0(Y)$. ${\rm Lip}$ spaces are effectively
the special case of ${\rm Lip}_0$ spaces for which the greatest element
of the unit ball is a multiplicative identity.
\end{prop}

The ${\rm Lip}(X)$ construction is insensitive to distances greater than $2$
because if $\rho(p,q) > 2$ then the slope $\frac{|f(p) - f(q)|}{\rho(p,q)}$
is dominated by $\|f\|_\infty$. Thus the metric can be truncated at $2$
without affecting the norm on ${\rm Lip}(X)$. The idea of Proposition
\ref{lipvslip0} is that once this is done, we can attach a base point
whose distance to every point of $X$ is $1$ and extend any Lipschitz
function on $X$ to be zero at the base point. The Lipschitz number of the
extended function now includes the slopes $\frac{|f(p) - f(e)|}{\rho(p,e)} =
|f(p)|$, whose collective contribution to this Lipschitz number exactly
matches the sup norm of the original function. Thus ${\rm Lip}$ spaces are
effectively the special case of ${\rm Lip}_0$ spaces which arises when
the base point is exactly one unit away from every other point. This
condition is characterized by the property that the greatest element of
the unit ball is the function which is $0$ at the base point and constantly
$1$ elsewhere, i.e., a multiplicative identity.

Every ${\rm Lip}_0$ space is a dual Banach space. This was first shown by
Arens and Eells in \cite{AE}, but was later rediscovered in varying
degrees of generality by a number of authors. (I know of four essentially
different methods of constructing the predual.) In \cite{W} I termed the
natural preduals ``Arens-Eells spaces'', but in their celebrated paper
\cite{GK} Godefroy and Kalton renamed them ``Lipschitz-free spaces''.
As a result these spaces now go by two different names in the literature.

The most important feature of the Arens-Eells spaces is a universal
property: for any pointed metric space $X$ there is a natural isometric
embedding $\iota: X \to {\rm \AE}(X)$ of $X$ in its Arens-Eells space,
which takes the base point of $X$ to the origin of ${\rm \AE}(X)$, and
which is universal among all nonexpansive base point preserving maps from
$X$ into any Banach space. This fact is due to me and first appeared in
\cite{W}, although it could have been deduced from results in an earlier
paper of Kadets \cite{K} (which I unfortunately was not aware of). Indeed,
one can simply define ${\rm \AE}(X)$ to be this universal space, whose
existence and uniqueness is routine, and then easily check that its dual
space is isometrically
isomorphic to ${\rm Lip}_0(X)$. This was pointed out in \cite{DTZ}.

In \cite{W} I asked whether the predual of ${\rm Lip}_0(X)$ is unique.
Since then this problem has been noted by various authors, but the only
previous work in this direction I know of is in \cite{Gd}, where it is
shown that ${\rm Lip}_0(X) \cong L^\infty[0,1]$ when $X$ is a separable
metric tree. This implies that the predual of ${\rm Lip}_0(Y)$ is unique
if $Y$ is isometrically contained in such a space. In this paper I show that for any $X$ the predual of ${\rm Lip}(X)$ is unique.

\section{${\rm Lip}$ spaces}

The main fact about ${\rm \AE}(X)$ we need is the following.

\begin{prop}\label{weakstar}
(\cite[Theorem 2.2.2]{W}) Let $X$ be a pointed metric space and let
$(f_\lambda)$ be a bounded net in ${\rm Lip}_0(X)$. Then $f_\lambda \to f
\in {\rm Lip}_0(X)$ weak* (relative to the predual ${\rm \AE}(X)$) if and only if
$f_\lambda \to f$ pointwise.
\end{prop}

Give ${\rm Lip}_0(X)$ the pointwise partial order, i.e., $f \leq g$ if
$f(p) \leq g(p)$ for all $p \in X$. Then the pointwise join or meet of any
two Lipschitz functions is again Lipschitz; indeed, it is easy to see that
the join or meet of any bounded family of Lipschitz functions is again
Lipschitz \cite[Proposition 1.5.5]{W}. This lattice structure is
related to the weak* topology.

\begin{lemma}\label{latticegenlemma}
Let $X$ be a pointed metric space, let $(f_\lambda)$ be a bounded net in
${\rm Lip}_0(X)$, and let $f \in {\rm Lip}_0(X)$. Suppose that
$f_\lambda \to f$ weak* relative to the predual ${\rm \AE}(X)$. Then
$f = \bigvee_\lambda \bigwedge_{\kappa \geq \lambda} f_\kappa$.
\end{lemma}

\begin{proof}
Since $(f_\lambda)$ is bounded,
$g = \bigvee_\lambda \bigwedge_{\kappa \geq \lambda} f_\kappa$ is a
well-defined element of ${\rm Lip}_0(X)$. Now given $p \in X$ and
$\epsilon > 0$, we can find $\lambda_0$ such that
$|f_\kappa(p) - f(p)| \leq \epsilon$ for all $\kappa \geq \lambda_0$.
Thus $\bigwedge_{\kappa \geq \lambda} f_\kappa \leq f(p) + \epsilon$
for any $\lambda$, and $\bigwedge_{\kappa \geq \lambda_0} f_\kappa \geq
f(p) - \epsilon$. This shows that $|g(p) - f(p)| \leq \epsilon$. Since
$p$ and $\epsilon$ were arbitrary, we conclude that $g = f$.
\end{proof}

The next result on our way to the uniqueness results is also of independent
interest. Say that a linear functional $\phi \in {\rm Lip}_0(X)^*$ is
{\it positive}, and write $\phi \geq 0$, if $f \geq 0$ implies
$\phi(f) \geq 0$. Say that $\phi$ is {\it normal} if it satisfies
$\phi(f_\lambda) \to \phi(f)$ whenever $(f_\lambda)$ is a bounded increasing
net in ${\rm Lip}_0(X)$ and $f = \bigvee f_\lambda$.

\begin{theo}\label{normalweakstarthm}
Let $X \in \mathcal{M}_0$ and suppose $\phi \in {\rm Lip}_0(X)^*$ is
positive. Then $\phi$ is weak* continuous relative to the predual
${\rm \AE}(X)$ if and only if it is normal.
\end{theo}

\begin{proof}
The forward direction is easy because if $(f_\lambda)$ is a bounded
increasing net then $f_\lambda \to \bigvee f_\lambda$ boundedly pointwise
and hence weak*
(Proposition \ref{weakstar}). For the reverse direction, assume $\phi$ is
normal; since the intervals of the form $(a,\infty)$ and $(-\infty,a)$
generate the topology on $\mathbb{R}$ it will suffice to show that for
every $a \in \mathbb{R}$ the sets $\phi^{-1}((-\infty, a])$ and
$\phi^{-1}([a,\infty))$ are weak* closed. We just consider the first
case, as the second is similar. By the Krein-Smulian theorem, it will
suffice to show that $\phi^{-1}((-\infty,a])$ contains the limits of all
bounded weak* convergent nets.

Thus let $(f_\lambda) \subseteq \phi^{-1}((-\infty,a])$ be bounded and suppose
it converges weak* to $f \in {\rm Lip}_0(X)$. By Lemma \ref{latticegenlemma}
we have $f = \bigvee_\lambda \bigwedge_{\kappa \geq \lambda} f_\kappa$. Then
$\phi(f_\lambda) \leq a$ for all $\lambda$, so by positivity we know that
$\phi(\bigwedge_{\kappa \geq \lambda} f_\kappa) \leq a$ for all $\lambda$.
But by normality
$$\phi\left(\bigwedge_{\kappa \geq \lambda} f_\kappa\right) \to \phi(f),$$
so that $\phi(f) \leq a$ as well. Thus $f \in \phi^{-1}((-\infty,a])$,
as desired.
\end{proof}

Now we prove the uniqueness result for ${\rm Lip}(X)$. The proof follows
the strategy used by Sakai to prove uniqueness of von Neumann algebra
preduals \cite{S}. The key property of ${\rm Lip}(X)$ used in the proof,
whose analog in other ${\rm Lip}_0$ spaces fails, is the fact
that if $f \in {\rm Lip}(X)$ satisfies $f \geq 0$ and $\|f\|_L \leq 1$,
then $\|f - 1_X\|_L \leq 1$ also. Another way to say this is that the
unit ball $[{\rm Lip}(X)]_1$ and its positive part $[{\rm Lip}(X)]_1^+$
are related by
$$[{\rm Lip}(X)]_1^+ \subseteq [{\rm Lip}(X)]_1 + 1_X.$$

A dual Banach space $V^*$ has a {\it unique predual} if $V^* \cong W^*$
implies $V \cong W$, for any Banach space $W$. An a priori stronger
condition is that every isometric isomorphism between $V^*$ and another
dual space $W^*$ is weak* continuous. (It would then have to be a weak*
homeomorphism and its adjoint would take $W \subseteq W^{**}$ isometrically
onto $V \subseteq V^{**}$.) In this case $V^*$ is said to have
a {\it strongly unique predual}. It is not known whether there are any
spaces whose predual is unique but not strongly unique.

\begin{theo}\label{aeuniquethm}
Let $X$ be a metric space and let $Y$ be as in Proposition \ref{lipvslip0}.
Then ${\rm \AE}(Y)$ is the strongly unique Banach space predual of
${\rm Lip}(X)$.
\end{theo}

\begin{proof}
Let $V$ be a Banach space and suppose $\phi: {\rm Lip}(X) \to V^*$ is a
surjective isometric isomorphism. Write $\langle v,f\rangle$ for the pairing
of $v \in V$ with $f \in {\rm Lip}(X)$, i.e., $\langle v,f\rangle =
\phi(f)(v)$.

This pairing induces an alternate weak* topology on ${\rm Lip}(X)$, which
within this proof I will refer to as the ``a-weak*'' topology. The standard
weak* topology indueced by ${\rm \AE}(Y)$ will be termed the ``s-weak*''
topology. One thing we know about this
a-weak* topology is that $[{\rm Lip}(X)]_1$ is compact relative
to it. Moreover, the comment made just before the theorem yields
$$[{\rm Lip}(X)]_1^+ = ([{\rm Lip}(X)]_1 + 1_X) \cap [{\rm Lip}(X)]_1,$$
so that the positive part of the unit ball is an intersection of two
a-weak* compact sets and is therefore itself a-weak* compact. By the
Krein-Smulian theorem, this implies that ${\rm Lip}(X)^+$, the positive
cone of ${\rm Lip}(X)$, is a-weak* closed.
It then follows from a standard separation theorem for topological vector
spaces that for any $f \in {\rm Lip}(X)$, $f \not\geq 0$,
there exists an a-weak* continuous linear functional which separates $f$
from ${\rm Lip}(X)^+$. That is, there exists $v \in V$ such that
$\langle v,f\rangle < 0$ but $\langle v,g\rangle \geq 0$ for all
$g \in {\rm Lip}(X)^+$. This means that for any $f \not\geq 0$ there
is a positive linear functional $v \in V$, i.e., an element $v \in V^+$,
such that $\langle v,f\rangle < 0$. (Here the order on $V$ is defined via
its pairing with ${\rm Lip}(X)$.) Applying this to $-f$, we also see that for
any $f \neq 0$ there is a $v \in V^+$ such that $\langle v,f\rangle \neq 0$.

We do not know that $V^+$ spans $V$, and indeed this can fail; see Example
\ref{posnegex} below. However, it follows from the previous paragraph that
${\rm span}(V^+)$ is dense in $V$.

Now suppose $(f_\lambda)$ is a bounded increasing net in ${\rm Lip}(X)$
and let $f = \bigvee f_\lambda$. Fix a subnet $(f_{\lambda_\kappa})$
which converges a-weak*, say to $g \in {\rm Lip}(X)$. Then for any
$v \in V^+$ the values $\langle v,f_{\lambda_\kappa}\rangle$
increase to $\langle v, g\rangle$. But also, just by positivity, we
have $\langle v,f_{\lambda_\kappa}\rangle \leq \langle v,f\rangle$
for every $\kappa$ and every $v \in V^+$. So
$\langle v,g\rangle \leq \langle v,f\rangle$ for every $v \in V^+$.
By what we showed earlier about the abundance of positive elements
of $V$, this implies that $g \leq f$; but also, since
$\langle v,f_{\lambda_\kappa}\rangle \leq \langle v,g\rangle$ for all
$\kappa$ and $v$ we similarly have $f_{\lambda_\kappa} \leq g$ for all
$\kappa$, and
therefore $f = \bigvee f_{\lambda_\kappa} \leq g$. So $f = g$, and we
conclude that every a-weak* convergent subnet of $(f_\lambda)$ converges
a-weak* to $f$, which implies that $f_\lambda$ itself converges a-weak* to $f$.

Thus, the pairing of any element of $V^+$ with ${\rm Lip}(X)$ is
normal, and therefore $\phi^*: V^{**} \to {\rm Lip}(X)^*$ takes
$V^+ \subseteq V^{**}$ into ${\rm \AE}(Y) \subseteq {\rm Lip}(X)^*$ by
Theorem \ref{normalweakstarthm}. Since ${\rm span}(V^+)$ is dense in $V$,
it follows that $\phi^*(V) \subseteq {\rm \AE}(Y)$. This implies that
$\phi^*(V) = {\rm \AE}(Y)$ and we conclude that the predual is strongly unique.
\end{proof}

In connection with the previous proof, note that Arens-Eells
spaces in general are not spanned by their positive elements. Here
we use a construction of ${\rm \AE}(X)$ as the closed subspace of
${\rm Lip}_0(X)^*$ generated by the point evaluations $\delta_p: f
\mapsto f(p)$.

\begin{exam}\label{posnegex}
Work on the unit interval $[0,1]$ augmented by a base point as in
Proposition \ref{lipvslip0}; denote this space $[0,1]^e$. Let $m$ be the
sum of elementary
``molecules'' $m = \sum_{n=0}^\infty (\delta_{2^{-2n}} - \delta_{2^{-2n -1}})$.
Since $\|\delta_{2^{-2n}} - \delta_{2^{-2n -1}}\| = 2^{-2n-1}$ in
${\rm \AE}([0,1]^e)$, this sum is absolutely convergent.

Suppose we could write $m = m^+ - m^-$ with $m^+, m^- \in {\rm \AE}([0,1]^e)$
both positive. For each $n$ find a function $f_n \in {\rm Lip}[0,1]$
which satisfies $0 \leq f_n \leq 1$ and which takes the value 1 at the points
$1, 2^{-2}, \ldots, 2^{-2n}$, the value $0$ at the points
$2^{-1}, 2^{-3}, \ldots, 2^{-2n-1}$, and is
constantly $0$ on $[0,2^{-2n-1}]$. Then
$$\langle m^+, 1_{[0,1]}\rangle \geq \langle m^+, f_n\rangle
\geq \langle m, f_n\rangle = n+1$$
for any $n$. Thus $\langle m^+, 1_{[0,1]}\rangle$ cannot be finite, a
contradiction.
\end{exam}

\section{${\rm Lip}_0$ spaces}

If the diameter of $X$ is finite then ${\rm Lip}_0(X)$ is isometrically isomorphic to ${\rm Lip}_0(Y)$ where $Y$ is $X$ rescaled to have diameter at most $1$, and ${\rm Lip}_0(Y)$ is (isometrically) a complemented codimension-one subspace of ${\rm Lip}(Y)$. In a previous version of this manuscript I tried to use that fact to infer that ${\rm Lip}_0(Y)$, and hence ${\rm Lip}_0(X)$ for any finite diameter metric space $X$, has a strongly unique predual. However, the proof I gave was incorrect.

We can still make an inference from ``strong uniqueness for finite diameter metric spaces'' to ``strong uniqueness for complete convex
metric spaces''. A metric space $X$ is {\it convex} if for every distinct $p,q \in X$ there
exists a third distinct point $r$ such that
$\rho(p,q) = \rho(p,r) + \rho(r,q)$.
If $X$ is complete and convex, then for any distinct $p,q \in X$ there is
an isometric embedding of the interval $[0,a] \subset \mathbb{R}$ into $X$
which takes $0$ to $p$ and $a$ to $q$, where $a = \rho(p,q)$.

Given any pointed metric space $X$ and any closed subset $K \subseteq X$
containing the base point, let $\mathcal{I}(K) = \{f \in {\rm Lip}_0(X):
f|_K = 0\}$. This is a weak* closed subspace (relative to the predual
${\rm \AE}(X)$) and we have ${\rm Lip}_0(X)/\mathcal{I}(K) \cong
{\rm Lip}_0(K)$ \cite[Corollary 4.2.7]{W}.

The proof of the following theorem uses a result of Dixmier \cite{D} according to which, for any Banach space $V$, a closed subspace of $V^*$ is a predual
of $V$ in the natural way if and only if the weak topology it induces on $V$ makes the
unit ball $[V]_1$ compact Hausdorff. (Thus $V$ has a unique predual if and only if all
closed subspaces of $V^*$ with this property are isometrically isomorphic, and it has
a strongly unique predual if and only if $V^*$ has exactly one closed subspace with
this property.)

\begin{theo}\label{convex}
Let $X$ be a complete convex metric space. For each $n \in \mathbb{N}$ let $X_n$ be the closed ball of
radius $n$ about the base point in $X$, and suppose that ${\rm Lip}_0(X_n)$ has a strongly unique predual for all $n$. Then ${\rm Lip}_0(X)$ has a
strongly unique predual.
\end{theo}

\begin{proof}
Suppose $W \subseteq {\rm Lip}_0(X)^*$ satisfies Dixmier's criterion.
Then it is a predual of ${\rm Lip}_0(X)$ and gives rise to a weak*
topology. I claim that the subspaces $\mathcal{I}(X_n)$ are closed in
this topology; granting this, it follows that for each $n$ the space
$W_n = \{\phi \in W: \phi|_{\mathcal{I}(X_n)} = 0\} \subseteq
({\rm Lip}_0(X)/\mathcal{I}(X_n))^* \cong ({\rm Lip}_0(X_n))^*$
satisfies Dixmier's criterion for ${\rm Lip}_0(X_n)$ and therefore
equals ${\rm \AE}(X_n) \subseteq {\rm Lip}_0(X_n)^*$. As
$\bigcup_{n=1}^\infty {\rm \AE}(X_n)$ is a dense subspace of
${\rm \AE}(X)$, this implies that ${\rm \AE}(X) \subseteq W$ and
hence that the two spaces are equal, establishing strong uniqueness.

To prove the claim, fix $n$ and let $h \in {\rm Lip}_0(X)$ be the
function $h(p) = \min(\rho(p,e), n)$. We will show that
$f \in [{\rm Lip}_0(X)]_1$ belongs to $\mathcal{I}(X_n)$ if and only if
$L(f \pm h) \leq 1$. This implies that
$$[\mathcal{I}(X_n)]_1 = [{\rm Lip}_0(X)]_1 \cap
([{\rm Lip}_0(X)]_1 + h) \cap ([{\rm Lip}_0(X)]_1 - h),$$
and hence that the unit ball of $\mathcal{I}(X_n)$ is weak* closed
relative to any predual. By the Krein-Smulian theorem, this is enough.

Suppose $f \not\in \mathcal{I}(X_n)$. Then for some $p \in X_n$
either $f(p) > 0$ or $f(p) < 0$. In the former case, $L(f + h) > 1$
because the slope
$$\frac{|(f + h)(p) - (f + h)(e)|}{\rho(p,e)} =
\frac{f(p) + \rho(p,e)}{\rho(p,e)}$$
exceeds $1$, and in the latter case $L(f - h) > 1$ for a similar reason.

Conversely, suppose $f \in \mathcal{I}(X_n)$; we must show that
$L(f \pm h) \leq 1$. Fix $p,q \in X$. If $\rho(p,e), \rho(q,e) \geq n$
then $h(p) = h(q)$ and
$$\frac{|(f \pm h)(p) - (f \pm h)(q)|}{\rho(p,q)}
= \frac{|f(p) - f(q)|}{\rho(p,q)} \leq 1,$$
and if $\rho(p,e), \rho(q,e) \leq n$ then $f(p) = f(q) = 0$ and
$$\frac{|(f \pm h)(p) - (f \pm h)(q)|}{\rho(p,q)}
= \frac{|h(p) - h(q)|}{\rho(p,q)} \leq 1.$$
In the remaining case, say
$\rho(p,e) > n$ and $\rho(q,e) < n$. By completeness and convexity we can
find a point $r$ such that $\rho(p,q) = \rho(p,r) + \rho(r,q)$ and
$\rho(r,e) = n$, and so
\begin{eqnarray*}
\frac{|(f \pm h)(p) - (f\pm h)(q)|}{\rho(p,q)}
&\leq& \max\left(\frac{|(f\pm h)(p) - (f\pm h)(r)|}{\rho(p,r)},
\frac{|(f\pm h)(r) - (f\pm h)(q)|}{\rho(r,q)}\right)\cr
&=& \max\left(\frac{|f(p) - f(r)|}{\rho(p,r)},
\frac{|h(r) - h(q)|}{\rho(r,q)}\right)\cr
&\leq& \max(L(f),L(h)) = 1
\end{eqnarray*}
using the elementary inequality $\frac{b + d}{a + c} \leq
\max(\frac{b}{a},\frac{d}{c})$ for $a,c > 0$ and $b,d \geq 0$.
We conclude that $L(f \pm h) \leq 1$, as desired.
\end{proof}

\begin{coro}\label{banach}
Let $V$ be a Banach space. If ${\rm Lip}_0([V]_1)$ has a strongly unique
predual, then so does ${\rm Lip}_0(V)$.
\end{coro}

\end{document}